\documentclass[10pt,a4paper]{article}
\usepackage{amsthm,amsmath,amssymb,txfonts,tipa,amscd,bm,graphicx}
\usepackage[all]{xy}
\usepackage{arydshln}
\usepackage{titlesec}
\theoremstyle{definition}
\newtheorem{theo}{Theorem}[section]

\newtheorem{exa}{Example}

\newtheorem{lem}[theo]{Lemma}
\newtheorem{cor}[theo]{Corollary}
\title{On embeddings between outer automorphism groups of right-angled Artin groups}
\author{Shiro Imamura}

\AtEndDocument{\bigskip{\footnotesize%
\textsc{Graduate School of Mathematical Sciences, the University of Tokyo, 3-8-1 Komaba, Meguro-ku, Tokyo, 153-8914, Japan} \par  
\textit{E-mail address}: \texttt{imamura@ms.u-tokyo.ac.jp}
}}
\date{}
\begin{document}
\maketitle
\begin{abstract}
Using abelian coverings of Salvetti complexes, embeddings of outer automorphism groups of right-angled Artin groups (RAAGs) into outer automorphism groups of their particular characteristic subgroups are constructed. Virtual embeddings of outer automorphism groups of finitely generated groups having the unique root property into outer automorphism groups of their particular subgroups are also given. These results provide us with rich examples of (virtual) embeddings between outer automorphism groups of RAAGs. Finite subgroups of pure (outer) automorphism groups of RAAGs are also investigated. $p$-adic valuations and $\mathbb{Z}_p$-ranks of these groups are determined to establish some non-embeddability conditions.
\end{abstract}
\tableofcontents
\section*{Acknowledgement}
First and foremost the author wishes to express his gratitude to his supervisor, Takuya Sakasai. The author would also like to thank Dawid Kielak for helpful comments, especially on the proof of Lemma \ref{noP3}. The author was supported by the Program for Leading Graduate Schools, MEXT, Japan.
\section{Introduction}
\subsection{Background}
Right-angled Artin groups have recently attracted much interest in geometric group theory due to their actions on $\text{CAT}(0)$ cube complexes. They were first introduced by Baudisch \cite{B} and their theory was developed, for example, by Droms \cite{Dr1} \cite{Dr2} \cite{Dr3} under the name ``graph groups''.  It is famous that they have played an important role in the affirmative solution of the virtual Haken conjecture \cite{A}.

The outer automorphism group of a group $G$ is the quotient group $\text{Aut}(G)/\text{Inn}(G)$, where $\text{Aut}(G)$ denotes the automorphism group of $G$ and $\text{Inn}(G)$ denotes the inner automorphism group, the normal subgroup of $\text{Aut}(G)$ consisting of conjugations.
It is obvious that $\text{Aut}(G)$ is a subgroup of $\text{Aut}(G*H)$. However, when it comes to the outer case, it is not at all clear whether or not there exists an embedding of $\text{Out}(G)$ into $\text{Out}(G*H)$. There are some results about embeddings between outer automorphism groups of free groups: Bogopolski and Puga \cite{B&P} showed that $\text{Out}(F_n)$ can be embedded into $\text{Out}(F_{r^n(n-1)+1})$ if $r$ is odd and coprime to $n-1$. Using abelian coverings of graphs, Bridson and Vogtmann \cite{B&V} improved this result by showing the even case:
\begin{theo}\cite[Corollary A]{B&V}\label{embBV}
There exists an embedding $\text{Out}(F_n)\hookrightarrow\text{Out}(F_m)$ for any $m$ of the form $m=r^n(n-1)+1$ with $r>1$ coprime to $n-1$.
\end{theo}
They have also proven a virtual embedding of $\text{Out}(F_n)$ into $\text{Out}(F_{d(n-1)+1})$, i.e. the existence of a finite-index subgroup of $\text{Out}(F_n)$ which embeds into $\text{Out}(F_{d(n-1)+1})$: 
\begin{theo}\cite[Proposition 2.5]{B&V}\label{virBV}
For all positive integers $n$ and $d$, there exists a subgroup of finite index $\Gamma\subset\text{Out}(F_n)$ and a monomorphism $\Gamma\hookrightarrow\text{Out}(F_m)$, where $m=d(n-1)+1$.
\end{theo}

It is known that any finite subgroups of $\text{Aut}(F_n)$ or $\text{Out}(F_n)$ can be realized as a group of automorphisms of a (not necessarily simple) graph with fundamental group $F_n$ (for example \cite{C}). This fact allows us to classify finite subgroups of $\text{Aut}(F_n)$ or $\text{Out}(F_n)$, where $k_i\geq2$ and $m_i\geq1$. For example, finite subgroups of $\text{Aut}(F_n)$ and $\text{Out}(F_n)$ are classified in the following way by Khramtsov \cite{Kh} and Mazurov \cite{Ma}, respectively: 
\begin{theo}\cite[Theorem 2, Theorem 3]{Kh}\label{Kh}
A finite group $G$ can be embedded into $\text{Aut}(F_n)$ if and only if $G$ is a subgroup of $\prod_{i\in I}(S_{k_i}\wr S_{m_i})$ such that $\sum_{i\in I}(k_i-1)m_i\leq n$, where $k_i\geq2$ and $m_i\geq1$ are integers. When $G$ is an abelian group $\prod_{i\in I}\mathbb{Z}_{{p_i}^{e_i}}$, $G$ can be embedded into $\text{Aut}(F_n)$ if and only if $\sum_{i\in I}(p_i^{e_i}-p_i^{e_i-1})\leq n$.
\end{theo}
\begin{theo}\cite[Theorem]{Ma}\label{Ma}
A finite group $G$ can be embedded into $\text{Out}(F_n)$ if and only if one of the following three cases holds:
\begin{itemize}
\item $G$ is a subgroup of $\text{Aut}(F_n)$;
\item $G$ is a subgroup of $S_{n+1}\times\mathbb{Z}_2$;
\item $n=10$ and $G$ is a subgroup of the semi-direct product of a non-abelian group of period $3$ and order $27$ and the group of its automorphisms. 
\end{itemize}
\end{theo}
As for finite subgroups of general linear groups, the following result is classical:
\begin{theo}\cite{Minkowski}\label{Minkowski}
The least common multiple $M(n)$ of the orders of all finite groups of $n\times n$-matrices over $\mathbb{Q}$ is given by:
\[
M(n)=\prod_{p}p^{\sum_{k\geq0}\left\lfloor\frac{n}{p^k(p-1)}\right\rfloor}.
\]
\end{theo}

\subsection{Results of this paper}
The author generalized the results introduced in the previous subsection to those of right-angled Artin groups. As for (virtual) embeddings, methods in \cite{B&V} are made use of to give (virtual) embeddings of outer automorphism groups of right-angled Artin groups into outer automorphism groups of their particular subgroups. Finite subgroups of outer automorphism groups of right-angled Artin groups are also investigated to establish non-embeddability conditions. 

We start by introducing necessary notations related to graphs and groups in Subsection 1.3.

Section 2 is devoted to stating the definition and classical results about right-angled Artin groups and their automorphism groups.

In Section 3 we investigate virtual embeddings of  $\text{Out}(A_\Gamma)$, where $A_\Gamma$ is the RAAG whose underlying graph is $\Gamma$. First we state and prove Theorem 3.1, which claims the virtual embeddability of $\text{Out}(A_\Gamma)$ into $\text{Out}(L)$, where $L$ is a finite-index subgroup of $A_\Gamma$ and contains a normal subgroup $K$ whose quotient $A_\Gamma/K$ is centerless. Then, as a corollary, we provide a case such that $L$ has a presentation as a right-angled Artin group and its underlying graph is concretely given.

Section 4, in which we consider embeddings of $\text{Out}(A_\Gamma)$, has a similar structure to its predecessor. Subsection 4.1 states Theorem 4.1, which claims the embeddability of $\text{Out}(A_\Gamma)$ into $\text{Out}(N)$, where $N$ is a finite-index characteristic subgroup of $A_\Gamma$ satisfying several conditions depending on the shape of the underlying graph $\Gamma$. Subsection 4.2 is devoted to proving Theorem 4.1. In this subsection, we use a geometric approach: we investigate abelian coverings of the Salvetti complex $S_\Gamma$, whose universal abelian covering is naturally embedded into a Euclidean space. Subsection 4.3 deals with two corollaries, where $N$'s structure as a right-angled Artin group is given by using the Bass-Serre theory.

Section 5 is devoted to the opposite direction, that is, giving non-embeddability conditions. In this section, the main target is the pure (outer) automorphism group of $A_\Gamma$, defined in Subsection 5.1. In Subsection 5.2, the relationship between (outer) automorphism groups and pure (outer) automorphism groups is clarified. In Subsection 5.3 finite subgroups of pure (outer) automorphism groups are investigated and $p$-adic valuations and $\mathbb{Z}_p$-ranks are completely determined (Theorem \ref{finsub}). Some non-embeddability conditions are given as corollaries.

RAAGs interpolate between free abelian groups and free groups. In the case of free abelian groups, their (outer) automorphism groups are precisely the general linear groups over integers and embeddability problems between them are easy: $\text{GL}(m,\mathbb{Z})\hookrightarrow\text{GL}(n,\mathbb{Z})$ if and only if $m\leq n$. However, when it comes to the case of free groups, embeddability problems between their outer automorphism groups are difficult. (The necessary and sufficient condition of when $\text{Out}(F_m)$ can be embedded into $\text{Out}(F_n)$ is still open.) Hence embeddability problems between outer automorphism groups of RAAGs are more difficult. This paper offers ways to attack these problems though they are far from completely solved. 

This paper is a slightly shortened version of the author's master thesis.

\subsection{Notations and conventions}
A {\it graph} $\Gamma$ is a $1$-dimensional $\text{CW}$ complex. In the remainder of this paper, we assume that graphs are finite and simple, i.e., loops and multi-edges are not permitted. Here are some notations:
\begin{itemize}
\item $\text{V}(\Gamma)$ denotes the set of $0$-dimensional cells in $\Gamma$. An element of $\text{V}(\Gamma)$ is called a {\it vertex}.
\item $\text{E}(\Gamma)$ denotes the set of $1$-dimensional cells in $\Gamma$. An element of $\text{E}(\Gamma)$ is called an {\it edge}.
\item We say a subgraph (subcomplex) $\Lambda$ is {\it full} in $\Gamma$ if every edge $e\in \text{E}(\Gamma)$ whose boundary is contained in $\Lambda$ is also an element of $\text{E}(\Lambda)$.
\item We say a subgraph $\Lambda$ is {\it spanned} by $W\subset\text{V}(\Gamma)$ if it is the minimum full subgraph that contains $W$. 
\item $\text{lk}(v)$, the {\it link} of $v$, denotes the subgraph spanned by vertices adjacent to $v$.
\item $\text{st}(v)$, the {\it star} of $v$, denotes the subgraph spanned by $(\text{V}(\Gamma)\cap\text{lk}(v))\cup\{v\}$.
\item $\text{CC}(v)$ denotes the set of the connected components of $\Gamma-\text{st}(v)$.
\item $\text{Clique}(\Gamma)$ denotes the set of the {\it cliques} (complete subgraphs) in $\Gamma$.
\item $\text{Aut}(\Gamma):=\text{Homeo}(\Gamma)/\text{Homeo}_0(\Gamma)$ denotes the automorphism group of $\Gamma$, where $\text{Homeo}_0(\Gamma)$ denotes the group of self-homeomorphisms isotopic to the identity.
\item We define a preorder $\leq$ called {\it vertex domination} on $\text{V}(\Gamma)$: $v\leq w$ iff $\text{lk}(v)\subseteq\text{st}(w)$. When $v<w$, we say $w$ {\it dominates} $v$. This order defines an equivalence relation $\sim$ on $\text{V}(\Gamma)$: $v\sim w$ if and only if $v\leq w$ and $w\leq v$. We denote the class $v$ belongs to by $[v]$.
\item Given a full subgraph $\Lambda$ of $\Gamma$, $\Gamma^d_\Lambda$ is defined as follows:
\begin{displaymath}
\Gamma^d_\Lambda=\left(\mathbb{Z}_d\times\Gamma\right)/\sim \text{, where $(a,x)\sim(b,y)$ if and only if $x=y\in\Lambda$}.
\end{displaymath}
\end{itemize}
As for groups:
\begin{itemize}
\item $\mathbb{Z}_d$ denotes the cyclic group of order $d$.
\item $S_n$ denotes the symmetric group of degree $n$.
\item $F_n$ denotes the free group of rank $n$.
\item $G^{*n}$ denotes the free product of $n$ copies of $G$.
\item $\displaystyle \mathop{*}_{i\in I}G_i$ denotes the free product of groups in $\{G_i\}_{i\in I}$.
\item $G^{\text{ab}}$ denotes the abelianization $G/[G,G]$ of a group $G$, where $[G,G]$ denotes the commutator subgroup of $G$.
\item $\text{Aut}_H(G)$ denotes the set of automorphisms of a group $G$ which fix a subgroup $H<G$ as a set.
\item By ``the conjugation by $g\in G$'', we mean the automorphism $\gamma_g$ of $G$ which sends each $x\in G$ to $gxg^{-1}$.
\item The commutator $[g,h]$ denotes $ghg^{-1}h^{-1}$.
\end{itemize}

\section{Right-angled Artin groups}
\subsection{Right-angled Artin groups}
\subsubsection{Definition}
Given a graph $\Gamma$, one can define the {\it Right-Angled Artin Group} (abbreviated by {\it RAAG}) $A_\Gamma$ with the underlying graph $\Gamma$ by the following presentation:
\begin{displaymath}
A_\Gamma=\langle \text{V}(\Gamma)\mid vw=wv,\ \text{for all } \{v,w\}\in \text{E}(\Gamma)\rangle.
\end{displaymath}
Here are some examples of RAAG:
\begin{itemize}
\item $A_{K_n}\cong \mathbb{Z}^n$, where $K_n$ denotes the complete graph with $n$ vertices.
\item $A_{\bar{K_n}}\cong F_n$, where $\bar{K_n}$ denotes the null graph with $n$ vertices.
\item $A_{\Gamma*\Lambda}\cong A_{\Gamma}\times A_{\Lambda}$, where $\Gamma*\Lambda$ denotes the simplicial join of $\Gamma$ and $\Lambda$.
\item $A_{\Gamma\sqcup\Lambda}\cong A_{\Gamma}* A_{\Lambda}$.
\end{itemize}
RAAGs can be viewed as groups which interpolate between free groups and free abelian groups.
\subsubsection{Salvetti complexes}
Given a graph $\Gamma$, we define the Salvetti complex $S_\Gamma$ as follows:
\[
S_\Gamma=\left.\left(\bigcup_{C\in \text{Clique}(\Gamma)}\{(x_v)_{v\in\text{V}(\Gamma)}\in\mathbb{R}^{|\text{V}(\Gamma)|}\mid x_w\in\mathbb{Z},\ \text{for all } w\notin C\}\right)\, \right/\mathbb{Z}^{|\text{V}(\Gamma)|}.
\]
$S_\Gamma$ is obtained from tori corresponding to maximal elements in $\text{Clique}(\Gamma)$ by gluing them together appropriately. It is shown by Charney and Davis \cite{C&D} that $S_\Gamma$ is a finite $K(A_\Gamma,1)$ complex.
\subsubsection{Bi-orderability of RAAGs}
It is shown by Duchamp and Krob\cite{D&K} that RAAGs are bi-orderable: we say a group $G$ is {\it bi-orderable} if there exists a total order on $G$ which is compatible with left and right multiplications. The {\it unique root property} of $A_\Gamma$ is derived from the bi-orderability:
\begin{lem}\label{URP}
Let $n$ be a positive integer. Then, for elements $x$, $y$ in $A_\Gamma$, $x^n=y^n$ implies $x=y$.
\end{lem}
\begin{proof}
Assume on the contrary that $x\neq y$. We may assume $x<y$ without loss of generality. $x^k<y^k$ yields that $x^{k+1}=x^kx<x^ky<y^ky=y^{k+1}$. Hence we inductively have $x^n<y^n$, which is a contradiction.
\end{proof}
\begin{lem}\label{Aut embedding}
Let $A$ be a group having the unique root property. 
If $[A:N]<\infty$, then the restriction map $\phi\mapsto\phi|_N$ defines an embedding of $\text{Aut}_N(A)$ into $\text{Aut}(N)$.
\end{lem}
\begin{proof}
Define $i$ to be the index $[A:N]$ of $N$ in $A$. We have $x^i\in N$ for every $x\in A$ as the order of every element in the finite group $A/N$ divides $i$, the order of the group. Assume that $\phi$ is in the kernel of the restriction map. Then, since $\phi(x)^i=\phi(x^i)=x^i$, the unique root property of $A$ yields that $\phi(x)=x$.
\end{proof}
\subsection{Automorphism groups of RAAGs}
The automorphism group $\text{Aut}(A_\Gamma)$ of a RAAG $A_\Gamma$ is generated by the following four types of elements known as the {\it Laurence-Servatius generators} (vertices not mentioned are fixed):
\begin{itemize}
\item {\it Inversions}: $\iota_v$ sends a vertex $v$ to its inverse.
\item {\it Graph symmetries}: $\bar{\sigma}$ ($\sigma\in\text{Aut}(\Gamma)$) sends each $v$ to $\sigma(v)$.
\item {\it (Left) Transvections}: $\lambda_{v,w}$ ($v\geq w$, $v\neq w$) sends $w$ to $vw$.
\item {\it Partial conjugations}:  $\gamma_{v,A}$ ($A\in \text{CC}(v)$) sends each $w$ in $A$ to $vwv^{-1}$.
\end{itemize}
Laurence \cite{L} conjectured and partially proved this fact and Servatius \cite{S} completed the proof.
The Laurence-Servatius generators consist of finite number of automorphisms, thus $\text{Aut}(A_\Gamma)$ is finitely generated. In fact, it is finitely presentable. Day\cite{D} has given a finite presentation, using a method called peak reduction. Before stating the result of Day, we introduce the notion of {\it Whitehead automorphisms}, by which Day gives the presentation.

Let $V$ denote the vertex set $\text{V}(\Gamma)$, and let $\bar{a}$ be $V\cap\{a, a^{-1}\}$ ($a\in V\cup V^{-1}$). A Whitehead automorphism is an element $\alpha\in\text{Aut}(A_\Gamma)$ of one of the following two types:
\begin{itemize}
\item[(1)] $\alpha$ restricted to $V\cup V^{-1}$ is a permutation of $V\cup V^{-1}$
\item[(2)] There is an element $a\in V\cup V^{-1}$, called the {\it multiplier} of $\alpha$, such that $\alpha(a)=a$ and for each $v\in V$, the element $\alpha(v)$ is in $\{v, va, a^{-1}v, a^{-1}va\}$.
\end{itemize}

There is a special notation for type (2) Whitehead automorphisms. Let $a\in A\subset L=V\cup V^{-1}$ such that $a^{-1}\notin A$. The symbol $(A, a)$ denotes the Whitehead automorphism satisfying 
$$
(A, a)(a)=a
$$
and for $v\in V-\bar{a}$,
$$
(A, a)(v)=\left\{
\begin{array}{ll}
v & \text{if } v\notin A \text{ and } v^{-1}\notin A \\ 
va & \text{if } v\in A \text{ and } v^{-1}\notin A \\ 
a^{-1}v & \text{if } v\notin A \text{ and } v^{-1}\in A \\
a^{-1}va & \text{if } v\in A \text{ and } v^{-1}\in A 
\end{array}
\right.
$$
if such automorphism exists. Say that $(A, a)$ is {\it well-defined} if the above formula defines an automorphism of $A_\Gamma$. Assume that every $(A, a)$ appearing in this paper is well-defined though we do not mention it in each case. A necessary and sufficient condition for $(A,a)$ to be well-defined is known:
\begin{lem}\cite[Lemma 2.5]{D}
For $A\subset L$ with $a\in A$ and $a^{-1}\notin A$, the automorphism $(A,a)$ is well-defined if and only if both of the following hold:
\begin{enumerate}
\item The set $V\cap A\cap A^{-1}-\text{lk}(\bar{a})$ is a union of connected components of $\Gamma-\text{st}(\bar{a})$
\item For each $x\in A-A^{-1}$, we have $a\geq x$.
\end{enumerate}
\end{lem}
Now we state the result of Day \cite{D}.
\begin{lem}\cite[Theorem A, Theorem 2.7]{D}\label{Day}
$$
\text{Aut}(A_\Gamma)=\langle\Omega\mid R\rangle,
$$
where $\Omega$ is the set of Whitehead automorphisms and the set of relations $R$ consists of:
\begin{enumerate}
\item $(A, a)^{-1}=(A-a+a^{-1}, a^{-1})$.
\item $(A, a)(B, a)=(A\cup B, a)$, where $A\cap B=\{a\}$.
\item $[(A, a),(B, b)]=e$, where $a^{\pm 1}\notin B$, $b^{\pm 1}\notin A$ and at least one of (a) $A\cap B=\emptyset$ or (b) $\bar{b}\in\text{lk}(\bar{a})$ holds.
\item $[(A, a),(B, b)]=(B-b+a, a)^{-1}$, where $a^{\pm 1}\notin B$, $b\notin A$, $b^{-1}\in A$ and at least one of (a) $A\cap B=\emptyset$ or (b) $\bar{b}\in\text{lk}(\bar{a})$ holds.
\item $(A-a+a^{-1}, b)(A, a)=(A-b+b^{-1}, a)\sigma_{a,b}$, where $b\in A$ with $b^{-1}\notin A$, $b\neq a$, $b\sim a$ and $\sigma_{a,b}$ is the type (1) Whitehead automorphism with $\sigma_{a,b}(a)=b^{-1}$, $\sigma_{a,b}(b)=a$.
\item $\sigma(A, a)\sigma^{-1}=(\sigma(A), \sigma(a))$, where $\sigma$ is a type (1) Whitehead automorphism.
\item The entire multiplication table of the type (1) Whitehead automorphisms.
\item $(A, a)=(L-a^{-1}, a)(L-A,a^{-1})$.
\item $[(A, a),(L-b^{-1}, b)]=e$, where $b^{\pm 1}\notin A$.
\item $[(A, a),(L-b^{-1}, b)]=(L-a^{-1}, a)$, where $a\neq b\in A$ and $b^{-1}\notin A$.
\end{enumerate}
\end{lem}
\section{Virtual embeddings}
\begin{theo}\label{Virtual embedding}
Let $A$ be a finitely generated group having the unique root property, and let $K$ be a normal subgroup of $A$ such that the quotient $G=A/K$ is finite and centerless. Then $\text{Out}(A)$ can be virtually embedded into $\text{Out}(L)$, where $L$ is a subgroup of $A$ containing $K$.
\end{theo}
\begin{proof}
Define $B$ to be the group of elements in $\text{Aut}_K(A)$ acting trivially on $A/K$ and note that, since the center of $G=A/K$ is trivial, $B\cap\text{Inn}(A)$ is contained in the set of conjugations by elements in $K$, and hence in $L$. Details are as follows: assume that the conjugation by $a\in A$ is in $B$, then we have for every $t\in A$, $atKa^{-1}=tK$. Since $K$ is normal, we have $atK=taK$, which yields that $aK\in Z(A/K)$.

Notice that, as $L$ is a disjoint union of cosets of $K$ in $A_\Gamma$, $\phi\in B<\text{Aut}_K(A)$ leaves $L$ invariant. In other words, $B$ is a subgroup of $\text{Aut}_{L}(A)$. Now, Lemma 2.2 tells us that the restriction map $B<\text{Aut}_K(A)\rightarrow \text{Aut}(L);\phi\mapsto\phi|_L$ induces an embedding of $B/(B\cap\text{Inn}(A))<\text{Out}(A)$ into $\text{Out}(L)=\text{Aut}(L)/\text{Inn}(L)$.

To see this is a virtual embedding, we consider the index of $B$ in $\text{Aut}(A)$. First, $[\text{Aut}(A):\text{Aut}_K(A)]$ is finite since the $\text{Aut}(A)$-orbit of $K$ is a finite set: recall the fact that a finitely generated group has only finitely many subgroups of a fixed index. Second, $[\text{Aut}_K(A):B]$ is finite since $\text{Aut}(A/K)$, the automorphism group of a finite group $A/K$, is also finite. Thus we have:
$$
[\text{Aut}(A):B]=[\text{Aut}(A):\text{Aut}_K(A)][\text{Aut}_K(A):B]<\infty.
$$
Since $B/(B\cap\text{Inn}(A))$ is the image of $B$ by the surjective homomorphism $\text{Aut}(A)\twoheadrightarrow\text{Out}(A)$, we have $[\text{Out}(A):B/(B\cap\text{Inn}(A))]<\infty$, which completes the proof.
\end{proof}
Using Theorem \ref{Virtual embedding}, we establish virtual embeddings between outer automorphism groups of RAAGs.
\begin{cor}\label{Virtualout}
Let $d$ be a positive integer and $v$ be a vertex in $\Gamma$. Then $\text{Out}(A_\Gamma)$ can be embedded virtually into $\text{Out}(A_{\Gamma^d_{\text{st}(v)}})$.
\end{cor}
\begin{proof}
We may assume that $\Gamma-\text{st}(v)$ is not empty, hence has vertex $w$: assume that $\text{st}(v)$ coincides with $\Gamma$, then the claim is trivial since $\Gamma^d_\Gamma=\Gamma$.
From Dirichlet's theorem on arithmetic progressions, there exists a prime $p\equiv1\pmod{d}$. We apply Theorem 3.1 to the case where $L$ is the kernel of the composition map of $f\colon A_\Gamma\twoheadrightarrow \mathbb{Z}_p\rtimes_\alpha\mathbb{Z}_d=G$ and $\pi\colon G\twoheadrightarrow\mathbb{Z}_d$. Here, $f$ sends $v$ (respectively, $w$) to the generator of $\mathbb{Z}_d$ (respectively, that of $\mathbb{Z}_p$) and sends other vertices to zero.

We define the semi-direct product structure of $G$ by $\alpha\colon\mathbb{Z}_d\hookrightarrow\text{Aut}(\mathbb{Z}_p)\cong\mathbb{Z}_{p-1};1\in\mathbb{Z}_d\mapsto\frac{p-1}{d}\in\mathbb{Z}_{p-1}$, and hence, $G$ is centerless. Details are as follows: let $(a,b)\in\mathbb{Z}_p\rtimes_\alpha\mathbb{Z}_d$ be an element in the center of $G$. Then we have $(1,0)(a,b)=(a,b)(1,0)$ and $(0,1)(a,b)=(a,b)(0,1)$. These equations are rewrited to $(a+1,b)=(a+\alpha_b(1),b)$ and $(\alpha_1(a),b+1)=(a,b+1)$, respectively. Thus we have $a=b=0$, since the homomorphism $\alpha$ is injective and an automorphism $\phi$ of $\mathbb{Z}_p$ has no non-trivial fixed points unless it is the identity mapping.

The isomorphism $L\cong A_{\Gamma_{\text{st}(c)}^d}$ is shown by using the Reidemeister-Schreier procedure: we adopt $\{v^k\mid 0\leq k\leq d-1\}$ as a Schreier transversal for $L<A_\Gamma$. Then we have $\{w\mid w\in\text{st}(v)\}\cup\{w_k=v^kwv^{-k}\mid w\notin\text{st}(v),\ 0\leq k\leq d-1\}$ generates $L$ and the set of relations is $\{vw=wv\mid w\in\text{lk}(v)\}\cup\{w_kw'_k=w'_kw_k\mid w,w'\notin\text{lk}(v),\ 0\leq k\leq d-1\}$.
\end{proof}
Note that when the underlying graph $\Gamma$ is a null graph, this is Theorem \ref{virBV}: $\bar{K_n}_{\text{st}(v)}^d=\bar{K}_{d(n-1)+1}$.
Note also that by applying Corollary \ref{Virtualout} repeatedly, various virtual embeddings of $\text{Out}(A_\Gamma)$ are constructed.
\section{Embeddings}
\subsection{Statement of the theorem}
\begin{theo}\label{Embedding}
Let $A_\Gamma$ be a centerless RAAG and let $K$ be the kernel of the homomorphism $A_\Gamma \twoheadrightarrow \prod_{v\in\text{V}(\Gamma)}\mathbb{Z}_{r_v}$($r_v\in\mathbb{Z}_{>0}$): this map sends $v$ to $1\in\mathbb{Z}_{r_v}$. Given a characteristic subgroup $N$ of $A_\Gamma$ containing $K$, $\text{Out}(A_\Gamma)$ can be embedded into $\text{Out}(N)$ if there exist integers $s_{v,A}$ ($v\in\text{V}(\Gamma)$, $A\in \text{CC}(v)$) satisfying the following five conditions:
\begin{enumerate}
\item $\sum_{A\in\text{CC}(v)}s_{v,A}=\sum_{B\in\text{CC}(w)}s_{w,B}$ for all $v\leq w$.
\item $s_{v,A}=\sum_{B\subseteq A}s_{w,B}$ for all $v\leq w$ and $A(\in\text{CC}(v))$ not containing $w$.
\item $s_{v,A}=s_{\sigma(v),\sigma(A)}$ for all $\sigma\in\text{Aut}(\Gamma)$.
\item $s_{w,\{v\}}$ does not depend on $w$ and is equal to $0$ if there exists a vertex $w'$ such that $v\leq w'$ and $\{v,w'\}\in\text{E}(\Gamma)$. 
\item $\sum_{A\in\text{CC}(v)}s_{v,A}\equiv-1 \pmod{r_v}$ for all $v\in\text{V}(\Gamma)$.
\end{enumerate}
\end{theo}
We define $s_{v}$ to be $s_{w,\{v\}}$ if there exists a vertex $w$ not adjacent to $v$ such that $w\geq v$ and $0$ if there exists no such vertices. Note that, if $v\sim v'$, then $s_v=s_{v'}$ from the third and the fourth conditions in Theorem \ref{Embedding}. Hence we use $s_{[v]}$ instead of $s_v$.
\subsection{Proof of Theorem \ref{Embedding}}
The proof of Theorem \ref{Embedding} consists of several steps.

First, we provide a sufficient condition in the form of the existence of a splitting of a short exact sequence. Then interpret the exact sequence geometrically. This is achieved by using a result of Bridson and Vogtmann \cite[Corollary 8.5]{B&V}. Confirm that the assumption of Lemma \ref{Vogtmann} is satisfied. Then what has to be done is to construct a splitting.

The construction of a splitting is as follows.
We first assign an image to each generator of the Laurence-Servatius generators. Then, by using Lemma \ref{Day}, we verify that the map can be extended to a group homomorphism defined on the whole $\text{Aut}(A_\Gamma)$. We also have to confirm that inner automorphisms are killed to show that the map in fact gives an embedding of $\text{Out}(A_\Gamma)$.

\subsubsection{Sufficient condition and its geometric interpretation}
Since $N$ is characteristic in $A_\Gamma$, we obtain that $\text{Aut}(A_\Gamma)=\text{Aut}_{N}(A_\Gamma)$ injects into $\text{Aut}(N)$ from Lemma \ref{Aut embedding}. Thus, in order to prove Theorem \ref{Embedding}, it is enough to show that the short exact sequence
\begin{displaymath}
\{e\}\rightarrow A_\Gamma/N\hookrightarrow\text{Aut}(A_\Gamma)/N\twoheadrightarrow\text{Out}(A_\Gamma)\rightarrow\{e\}
\end{displaymath}
splits.
Here, we naturally identify $\text{Inn}(A_\Gamma)$ and its subgroup $\{\text{(the conjugation with }g)\in\text{Aut}(A_\Gamma)\mid g\in N\}\cong\text{Inn}(N)$ with $A_\Gamma$ and $N$.

The sufficiency of the existence of a splitting is obtained from the sequence of injections:
\begin{displaymath}
\text{Out}(A_\Gamma)\hookrightarrow\text{Aut}(A_\Gamma)/N\hookrightarrow\text{Aut}(N)/N=\text{Out}(N).
\end{displaymath}

Before recalling a result by Bridson and Vogtmann \cite{B&V}, we introduce some notations:
\begin{itemize}
\item $\text{Deck}(\hat{X},X)$ denotes the covering transformation group of the covering $p\colon \hat{X}\rightarrow X$. That is, $\text{Deck}(\hat{X},X)=\{f\colon\hat{X}\rightarrow\hat{X}\mid p\circ f=p\}$.
\item $\text{fhe}(\hat{X})$ denotes the set of self-homotopy equivalences of $\hat{X}$ that are fiber-preserving, endowed with the compact-open topology: $\text{fhe}(\hat{X})=\{\hat{h}\colon\hat{X}\rightarrow\hat{X}\mid p(\hat{x})=p(\hat{y})\Rightarrow p\circ\hat{h}(\hat{x})=p\circ\hat{h}(\hat{y})\}$.
\item $\text{FHE}(\hat{X})$ denotes the group consisting of connected components of $\text{fhe}(\hat{X})$:
$$
\text{FHE}(\hat{X})=\pi_0(\text{fhe}(\hat{X})).
$$
\item $\text{he}(X)$ denotes the set of self-homotopy equivalences of $X$.
\item $\text{HE}(X)=\pi_0(\text{he}(X))$.
\item $\text{he}_{\bullet}(X)$($\subset\text{he}(X)$) denotes the set of base point-fixing homotopy equivalences.
\item $\text{HE}_{\bullet}(X)=\pi_0(\text{he}_{\bullet}(X))$.
\end{itemize}
The following gives us a geometric interpretation of the short exact sequence.
\begin{lem}\cite[Corollary 8.5]{B&V}\label{Vogtmann}
Let $X$ be a $K(\pi, 1)$ space and let $p\colon \hat{X}\rightarrow X$ be a covering space with $N=p_*\pi_1(\hat{X})$ characteristic in $\pi$. If the centralizer $Z_\pi(N)$ is trivial, then the following diagram of groups is commutative and the vertical maps are isomorphisms:
$$
\xymatrix{
\{e\} \ar[r] & \pi/N \ar[d]\ar[r]& \text{Aut}(\pi)/N \ar[d]\ar[r]&\text{Out}(\pi) \ar[d]\ar[r]&\{e\} \\
\{e\} \ar[r] & \text{Deck}(\hat{X},X)\ar[r]& \text{FHE}(\hat{X}) \ar[r]&\text{HE}(X)\ar[r]&\{e\}
}
$$
where $\pi/N\rightarrow\text{Aut}(\pi)/N$ is the map induced by the action of $\pi$ on itself by inner automorphisms.
\end{lem}
We apply this lemma to the following settings: $\pi=A_\Gamma$, $X=S_\Gamma$ and $p\colon\hat{X}\rightarrow X$ is the covering corresponding to the characteristic subgroup $N$ given in Theorem \ref{Embedding}. The assumption of the lemma can be confirmed as follows:

\noindent{\bf Proof of the triviality of the centralizer}
\begin{proof}
Let $g\in Z_{A_\Gamma}(N)$. For all $x\in A_\Gamma$, $x^{[A_\Gamma:N]}\in N$. Thus 
$$(gxg^{-1})^{[A_\Gamma:N]}=gx^{[A_\Gamma:N]}g^{-1}=x^{[A_\Gamma:N]}.
$$
From the unique root property of $A_\Gamma$, we have $gxg^{-1}=x$, which yields $x\in Z(A_\Gamma)=\{e\}$.
\end{proof}

\subsubsection{Construction of a splitting}
Let $U\subset\mathbb{R}^{|\text{V}(\Gamma)|}$ be the universal abelian covering of $S_\Gamma$:
\begin{displaymath}
U=\bigcup_{C\in \text{Clique}(\Gamma)}\{(x_v)_{v\in\text{V}(\Gamma)}\mid x_w\in\mathbb{Z},\ \text{for all } w\notin C\}.
\end{displaymath}
Since $p\colon\hat{X}\rightarrow S_\Gamma$ is an abelian covering, $U$ lies over $\hat{X}$: $\hat{X}\cong U/\text{Deck}(U,\hat{X})$. 
We construct the splitting as follows:
\begin{enumerate}
\item Lift each element $\phi$ in the set of the Laurence-Servatius generators to the {\it standard lift} $\tilde{\phi}\colon U\rightarrow U$, which fixes the origin.
\item {\it Shift} the standard lift $\tilde{\phi}$ in the direction of $x_v$ by length $s_{v,A}$ (respectively, $s_{[v]}$) when $\phi$ is a partial conjugation $\gamma_{v,A}$ (respectively, when $\phi = \iota_v$). We do not shift other Laurence-Servatius generators.
\item Verify that the lifts satisfy Day's presentation (Lemma \ref{Day}).
\item Confirm that each lift of each inner automorphism is homotopic to an element in $\text{Deck}(U,\hat{X})$.
\end{enumerate}
As $S_\Gamma$ is a $K(A_\Gamma,1)$ complex, we can naturally identify $\text{Aut}(A_\Gamma)$ with $\text{HE}_{\bullet}(S_\Gamma)$. The lift $\hat{\phi}$ of $\phi\in\text{Aut}(A_\Gamma)$ sends a lattice point $\bm{x}\in\mathbb{Z}^{|\text{V}(\Gamma)|}$ to $\bm{\Phi}\bm{x}+\bm{s}$. Here, $\bm{\Phi}\in\text{GL}(|\text{V}(\Gamma)|,\mathbb{Z})$ is the matrix corresponding to the action of $\phi$ on $\text{H}_1({A_\Gamma})$ and $\bm{s}$ is the shift of $\hat{\phi}$ from the standard lift. What we only have to do is to confirm that the affine transformations corresponding to the Laurence-Servatius generators satisfy the presentation.
\subsubsection{Relations are satisfied}
Notice that, as $\bm{\Phi}$ is the matrix induced from the action of $\phi\in\text{Aut}(A_\Gamma)$ on $A_\Gamma^\text{ab}$, linear parts of transformations appearing in both side of the relations automatically coincide. 

Before beginning the computation, we offer some notations and formulas which connect the Laurence-Servatius generators with Whitehead automorphisms:
\begin{enumerate}
\item The set of type (1) Whitehead automorphisms coincides with the subgroup of $\text{Aut}(A_\Gamma)$ generated by inversions and graph symmetries, which is isomorphic to the semi-direct product $\mathbb{Z}_2^{|\text{V}(\Gamma)|}\rtimes\text{Aut}(\Gamma)$. ($\text{Aut}(\Gamma)<S_{|\text{V}(\Gamma)|}$ acts on $\mathbb{Z}_2^{|\text{V}(\Gamma)|}$ by basis permutations.)
\item $(\{v,w\},v)=\rho_{v,w}$.
\item $(\{v,w^{-1}\},v)=\lambda_{v,w}^{-1}$.
\item $(\{v^{-1},w\},v^{-1})=\rho_{v,w}^{-1}$.
\item $(\{v^{-1},w^{-1}\},v^{-1})=\lambda_{v,w}$.
\item $(\{v\}\cup A\cup A^{-1},v)=\gamma_{v,A}^{-1}$.
\item $(\{v^{-1}\}\cup A\cup A^{-1},v)=\gamma_{v,A}$.
\item $(L-v^{-1},v)=\gamma_v^{-1}=\prod_{A\in\text{CC}(v)}\gamma_{v,A}^{-1}$.
\item $(L-v,v^{-1})=\gamma_v=\prod_{A\in\text{CC}(v)}\gamma_{v,A}$.
\item $\rho_{v,w}$ denotes an automorphism called a {\it right transvection} which sends $w$ to $wv$. $\rho_{v,w}=\lambda_{v,w}$ if $w$ is adjacent to $v$ and $\rho_{v,w}=\lambda_{v,w}\gamma_{v,\{w\}}^{-1}$ otherwise. (Notice that $\{w\}\in\text{CC}(v)$ if $v\geq w$ and $v$ is not adjacent to $w$.)
\item We define $S_v$ to be $\sum_{A\in\text{CC}(v)}s_{v,A}$, the length of the shift of $\gamma_v$.
\item $\bm{I}$ denotes the identity matrix acting trivially on $\mathbb{R}^{|\text{V}(\Gamma)|}$.
\item $\bm{E}_{vw}$ denotes the matrix which sends $\bm{x}_u$ to $\delta_{uw}\bm{x}_v$, where $\bm{x}_v$ denotes the basis vector of $\mathbb{R}^{|\text{V}(\Gamma)|}$ corresponding to $v\in\text{V}(\Gamma)$.
\end{enumerate}
\noindent{\bf The second relation} Since $A$ and $B$ are unions of elements in $\{\{b\}\subset V^{\pm}\mid \bar{b}\leq \bar a\}\cup \text{CC}(\bar{a})$, it is enough to show that, for each $v\in V$, elements in $\{\lambda_{v,w}\mid w\leq v\}\cup\{\gamma_{v,A}\mid A\in\text{CC}(v)\}$ mutually commute. This can be confirmed as follows:
\begin{align*}
\left(\bm{I}+\bm{E}_{vw}\right)\left(\bm{I}+\bm{E}_{vw'}\right)\bm{x}&=\left(\bm{I}+\bm{E}_{vw'}\right)\left(\bm{I}+\bm{E}_{vw}\right)\bm{x}\\
\left(\bm{I}+\bm{E}_{vw}\right)(\bm{I}\bm{x}+s_{v,A}\bm{x}_v)&=\bm{I}((\bm{I}+\bm{E}_{vw})\bm{x})+{s}_{v,A}\bm{x}_v\\
\bm{I}(\bm{I}\bm{x}+s_{v,A}\bm{x}_v)+s_{v,A'}\bm{x}_v&=\bm{I}(\bm{I}\bm{x}+s_{v,A'}\bm{x}_v)+s_{v,A}\bm{x}_v
\end{align*}
The first equality shows $\lambda_{v,w}\lambda_{v,w'}=\lambda_{v,w'}\lambda_{v,w}$, the second shows $\lambda_{v,w}\gamma_{v,A}=\gamma_{v,A}\lambda_{v,w}$ and the third shows $\gamma_{v,A}\gamma_{v,A'}=\gamma_{v,A'}\gamma_{v,A}$.

\noindent{\bf The first relation} This is clear from the above formulas and the second relation.

\noindent{\bf The third relation} It is enough to show that one element in $\{\lambda_{v,v'}|v'\leq v\}\cup\{\gamma_{v,A}|A\in\text{CC}(v)\}$ and another element in $\{\lambda_{w,w'}|w'\leq w\}\cup\{\gamma_{w,B}|B\in\text{CC}(w)\}$ commute. This is confirmed as follows:
\begin{align*}
(\bm{I}+\bm{E}_{vv'})(\bm{I}+\bm{E}_{ww'})\bm{x}&=(\bm{I}+\bm{E}_{ww'})(\bm{I}+\bm{E}_{vv'})\bm{x}, \\
(\bm{I}+\bm{E}_{vv'})(\bm{I}\bm{x}+s_{w,B}\bm{x}_{w})&=\bm{I}(\bm{I}+\bm{E}_{vv'})\bm{x}+s_{w,B}\bm{x}_{w}, \\
\bm{I}(\bm{I}\bm{x}+s_{w,B}\bm{x}_{w})+s_{v,A}\bm{x}_{v}&=\bm{I}(\bm{I}\bm{x}+s_{v,A}\bm{x}_{v})+s_{w,B}\bm{x}_{w}.
\end{align*} 

\noindent{\bf The fourth relation} From the second relation, we have $(A,a)=(A-b^{-1},a)(\{a,b^{-1}\},a)=(\{a,b^{-1}\},a)(A-b^{-1},a)$.
Thus, using the third relation the left hand side of the relation becomes:
\begin{align*}
[(A,a),(B,b)]&=(A,a)(B,b)(A,a)^{-1}(B,b)^{-1}\\
&=(\{a,b^{-1}\},a)(A-b^{-1},a)(B,b)(A-b^{-1},a)^{-1}(\{a,b^{-1}\},a)^{-1}(B,b)^{-1}\\
&=[(\{a,b^{-1}\},a),(B,b)].
\end{align*}
Hence it is enough to show that, for $\bar{b}\leq\bar{a}$, $[(\{a,b^{-1}\},a),(B,b)]=(B-b+a,a)^{-1}.$ In the words of the Laurence-Servatius generators, it suffices to show that, for $u\leq v\leq w$ and for $w\notin A\in\text{CC}(v)$, $[\lambda_{w,v}^{-1},\gamma_{v,A}]=\left(\prod_{B\subseteq A}\gamma_{w,B}\right)^{-1}$ and $[\lambda_{w,v}^{-1},\lambda_{v,u}]=\lambda_{w,u}^{-1}$ hold:
\begin{align*}
(\bm{I}+\bm{E}_{wv})^{-1}((\bm{I}+\bm{E}_{wv})(\bm{x}-s_{v,A}\bm{x}_v)+s_{v,A}\bm{x}_v)&=\bm{x}-\sum_{B\subseteq A}s_{w,B}\bm{x}_w,\\
(\bm{I}+\bm{E}_{wv})^{-1}(\bm{I}+\bm{E}_{vu})(\bm{I}+\bm{E}_{wv})(\bm{I}+\bm{E}_{vu})^{-1}\bm{x}&=(\bm{I}+\bm{E}_{wu})^{-1}\bm{x}.
\end{align*}
The second condition of Theorem \ref{Embedding} guarantees that these equalities hold.

\noindent{\bf The fifth relation} Set $A=C\sqcup\{a,b\}$. From the second relation, we have:
\begin{align*}
(A-a+a^{-1},b)&=(C+b,b)(\{b,a^{-1}\},b),\\
(A,a)&=(C+a,a)(\{a,b\},a),\\
(A-b+b^{-1},a)&=(C+a,a)(\{a,b^{-1}\},a).
\end{align*}
From the fourth relation, $[(\{b,a^{-1}\},b),(C+a,a)]=(C+b,b)^{-1}$. Hence the left hand side of the relation is:
\begin{align*}
(A-a+a^{-1},b)(A,a)&=(C+b,b)(\{b,a^{-1}\},b)(C+a,a)(\{a,b\},a) \\
&=(C+b,b)(C+b,b)^{-1}(C+a,a)(\{b,a^{-1}\},b)(\{a,b\},a) \\
&=(C+a,a)(\{b,a^{-1}\},b)(\{a,b\},a).
\end{align*}
On the other hand, the right hand side of the relation is as follows:
\[
(A-b+b^{-1},a)\sigma_{a,b}=(C+a,a)(\{a,b^{-1}\},a)\sigma_{a,b}.
\]
Thus it is enough to show that for $\bar{a}\sim\bar{b}$,
\[
(\{a^{-1},b\},b)(\{a,b\},a)=(\{a,b^{-1}\},a)\sigma_{a,b}.
\]
In the words of the Laurence-Servatius generators, this becomes: $\lambda_{w,v}^{-1}\rho_{v,w}=\lambda_{v,w}^{-1}\iota_w\tau_{vw}$, $\lambda_{w,v}\lambda_{v,w}^{-1}=\rho_{v,w}\tau_{vw}\iota_w$, $\rho_{w,v}\rho_{v,w}^{-1}=\lambda_{v,w}\iota_v\tau_{vw}$ and $\rho_{w,v}^{-1}\lambda_{v,w}=\rho_{v,w}^{-1}\tau_{vw}\iota_v.$ Here $\tau_{vw}$ denotes the transposition of $v$ and $w$ and $\bm{T}_{vw}$ denotes the corresponding matrix: $\bm{T}_{vw}=\bm{I}-\bm{E}_{vv}-\bm{E}_{ww}+\bm{E}_{vw}+\bm{E}_{wv}$. Note that the first equation is equivalent to the second, and the third is equivalent to the fourth. Hence we only have to confirm that the second and the fourth equations hold. Moreover, we assume that $v$ and $w$ are not adjacent since, if they are adjacent, there are no shifts and hence the claim is trivial. Now the two equations are confirmed as follows:
\begin{align*}
(\bm{I}+\bm{E}_{wv})(\bm{I}+\bm{E}_{vw})^{-1}\bm{x}&=\bm{I}(\bm{I}+\bm{E}_{vw})\bm{T}_{vw}((\bm{I}-2\bm{E}_{ww})\bm{x}+s_{[v]}\bm{x}_w)-s_{[v]}\bm{x}_v,\\
\bm{I}(\bm{I}+\bm{E}_{wv})^{-1}(\bm{I}+\bm{E}_{vw})\bm{x}+s_{[v]}\bm{x}_w&=\bm{I}(\bm{I}+\bm{E}_{vw})^{-1}\bm{T}_{vw}((\bm{I}-2\bm{E}_{vv})\bm{x}+s_{[v]}\bm{x}_v)+s_{[v]}\bm{x}_v.
\end{align*}

\noindent{\bf The sixth relation} It is enough to show that $\iota_v\gamma_{v,A}\iota_v=\gamma_{v,A}^{-1}$, $\iota_v\lambda_{v,w}\iota_v=\lambda_{v,w}^{-1}$, $\iota_w\lambda_{v,w}\iota_w=\rho_{v,w}^{-1}$, $\sigma\gamma_{v,A}\sigma^{-1}=\gamma_{\sigma(v),\sigma(A)}$ and $\sigma\lambda_{v,w}\sigma^{-1}=\lambda_{\sigma(v),\sigma(w)}$, where $\sigma$ is an automorphism of $\Gamma$. 
The third condition of Theorem \ref{Embedding} guarantees that the fourth and fifth equations hold. Other equations are confirmed as follows:
\begin{align*}
(\bm{I}-2\bm{E}_{vv})(\bm{I}((\bm{I}-2\bm{E}_{vv})\bm{x}+s_{[v]}\bm{x}_v)+s_{v,A}\bm{x}_v)+s_{[v]}\bm{x}_v&=\bm{I}\bm{x}-s_{v,A}\bm{x}_v, \\
(\bm{I}-2\bm{E}_{vv})(\bm{I}+\bm{E}_{vw})((\bm{I}-2\bm{E}_{vv})\bm{x}+s_{[v]}\bm{x}_v)+s_{[v]}\bm{x}_v&=(\bm{I}+\bm{E}_{vw})^{-1}\bm{x}, \\
(\bm{I}-2\bm{E}_{ww})(\bm{I}+\bm{E}_{vw})((\bm{I}-2\bm{E}_{ww})\bm{x}+s_{[w]}\bm{x}_w)+s_{[w]}\bm{x}_w&=(\bm{I}+\bm{E}_{vw})^{-1}\bm{x}+s_{[w]}\bm{x}_v.
\end{align*}

\noindent{\bf The seventh relation} It is enough to confirm the semi-direct product structure is preserved. This is confirmed as follows:
\[
\bm{\sigma}((\bm{I}-2\bm{E}_{vv})\bm{\sigma}^{-1}\bm{x}+s_{[v]}\bm{x}_v)=(\bm{I}-2\bm{E}_{\sigma(v)\sigma(v)})\bm{x}+s_{[\sigma(v)]}\bm{x}_{\sigma(v)}.
\]
This equation shows $\bar{\sigma}\iota_v\bar{\sigma}^{-1}=\iota_{\sigma(v)}$, where $\sigma\in\text{Aut}(\Gamma)$. ($\bm{\sigma}$ and $\bar{\sigma}$ denote the permutation matrix and the automorphism of $A_\Gamma$ corresponding to $\sigma$.)

\noindent{\bf The eighth relation} Note that this relation automatically holds as a consequence of the first and second relations.

\noindent{\bf The ninth relation} It is enough to show that, for each $v\in V$, $\gamma_v$ commutes with $\lambda_{w,u}$ and $\gamma_{w,A}$, where $v\neq u$ and $v\notin A$. This can be confirmed as follows:
\begin{align*}
(\bm{I}+\bm{E}_{wv})(\bm{I}\bm{x}+S_v\bm{x}_v)&=\bm{I}(\bm{I}+\bm{E}_{wv})\bm{x}+S_v\bm{x}_v,\\
\bm{I}(\bm{I}\bm{x}+S_v\bm{x}_v)+s_{w,A}\bm{x}_w&=\bm{I}(\bm{I}\bm{x}+s_{w,A}\bm{x}_w)+S_v\bm{x}_v.
\end{align*}

\noindent{\bf The tenth relation}
 Set $A=C\sqcup\{a,b\}$. From the second relation, we have $(A,a)=(C+a,a)(\{a,b\},a)$. From the ninth relation, $(C+a,a)$ commutes with $(L-b^{-1},b)$. Hence, it is enough to show that, for $\bar{b}\leq\bar{a}$,
\[
[(\{a,b\},a),(L-b^{-1},b)]=(L-a^{-1},a).
\]
In the words of the Laurence-Servatius generators, this becomes:
$
[\rho_{v,w},\gamma_w^{-1}]=\gamma_v^{-1}, 
[\lambda_{v,w}^{-1},\gamma_w]=\gamma_v^{-1}, 
[\rho_{v,w}^{-1},\gamma_w^{-1}]=\gamma_v, 
[\lambda_{v,w},\gamma_w]=\gamma_v.
$ These are confirmed as follows:
\begin{align*}
(\bm{I}+\bm{E}_{vw})(\bm{I}((\bm{I}-\bm{E}_{vw})(\bm{I}\bm{x}+S_w\bm{x}_w)+s_{v,\{w\}}\bm{x}_v)-S_w\bm{x}_w)-s_{v,\{w\}}\bm{x}_v&=\bm{I}\bm{x}-S_v\bm{x}_v,\\
(\bm{I}-\bm{E}_{vw})(\bm{I}(\bm{I}+\bm{E}_{vw})(\bm{I}\bm{x}-S_w\bm{x}_w)+S_w\bm{x}_w)&=\bm{I}\bm{x}-S_v\bm{x}_v,\\
(\bm{I}-\bm{E}_{vw})(\bm{I}((\bm{I}+\bm{E}_{vw})(\bm{I}\bm{x}+S_w\bm{x}_w)-s_{v,\{w\}}\bm{x}_v)-S_w\bm{x}_w)+s_{v,\{w\}}\bm{x}_v&=\bm{I}\bm{x}+S_v\bm{x}_v,\\
(\bm{I}+\bm{E}_{vw})(\bm{I}(\bm{I}-\bm{E}_{vw})(\bm{I}\bm{x}-S_w\bm{x}_w)+S_w\bm{x}_w)&=\bm{I}\bm{x}+S_v\bm{x}_v.
\end{align*}
Notice that, the first condition of Theorem \ref{Embedding} guarantees that $S_v=S_w$.

\subsubsection{Inner automorphisms are killed}
Since $\text{Deck}(U,\hat{X})=N/[A_\Gamma,A_\Gamma]>K/[A_\Gamma,A_\Gamma]=\prod_{v\in\text{V}(\Gamma)}r_v\mathbb{Z}\subset\mathbb{Z}^{|\text{V}(\Gamma)|}=A_\Gamma/[A_\Gamma,A_\Gamma]$, it is enough to show that, for every vertex $v$, the conjugation $\gamma_v$ by $v$ lifts to an element homotopic to an element in $\prod_{v\in\text{V}(\Gamma)}r_v\mathbb{Z}$. Notice that the shift of $\gamma_v$ is of length $\sum_{A\in\text{CC}(v)}s_{v,A}$ in the direction of $x_v$, since $\gamma_v=\prod_{A\in\text{CC}(v)}\gamma_{v,A}$.
Now, the fifth condition in Theorem \ref{Embedding} guarantees that $\gamma_v$ is killed since $\hat{\gamma_v}$ is homotopic to the translation of length $\sum_{A\in\text{CC}(v)}s_{v,A}+1$ in the direction of $x_v$, which is an element of $\prod_{v\in\text{V}(\Gamma)}r_v\mathbb{Z}$. Now the proof of Theorem \ref{Embedding} is completed.\qed

\subsection{Corollaries}
For a group $G$ and a positive integer $r$, let $P_r(G)$ be the normal closure of $[G,G]\cup\{g^r\mid g\in G\}$. 
Note that $P_r(G)$ is the kernel of the surjective homomorphism onto $G^{\text{ab}}\otimes\mathbb{Z}_r$, and hence, when $G^{\text{ab}}$ is finitely-generated, $P_r(G)$ is a finite-index characteristic subgroup of $G$.
\begin{cor}\label{FF}
Define $m$ to be the number of connected components of $\Gamma$. $\text{Out}(A_\Gamma)$ can be embedded into $\text{Out}(P_r(A_\Gamma))$ if $r$ is coprime to $m-1$.
\end{cor}
\begin{proof}
Let $s$ be an integer such that $s(m-1)\equiv-1\pmod{r}$, whose existence is guaranteed by the assumption that $r$ is coprime to $m-1$.
To show Corollary \ref{FF}, apply Theorem \ref{Embedding} to the following settings:
\begin{itemize}
\item $r_v=r$,
\item $s_{v,A}=s$, if $A$, a connected component of $\Gamma-\text{st}(v)$, is also a connected component of $\Gamma$.
\item $s_{v,A}=0$, otherwise. 
\end{itemize}
These satisfy the conditions in Theorem \ref{Embedding}.
\end{proof}
Notice that this corollary is a natural generalization of Theorem \ref{embBV}: by using the Euler characteristic argument, we can see $P_r(F_n)$ is isomorphic to $F_{r^n(n-1)+1}$.

Let $\phi_i\colon G_i\rightarrow A_i$ ($i\in\{1,2\}$) be a homomorphism onto a finite abelian group $A_i$. Define $N_i$ to be the kernel of this map.
\begin{lem}\label{FP}
The kernel of $\phi_1\times\phi_2\colon G_1\times G_2\rightarrow A_1\times A_2$ (respectively, $\phi_1*\phi_2\colon G_1*G_2\rightarrow A_1\times A_2$) is isomorphic to $N_1\times N_2$ (respectively, $N_1^{*|A_2|}*N_2^{*|A_1|}*F_{(|A_1|-1)(|A_2|-1)}$). 
\end{lem}
\begin{proof}
The first half of the claim is obvious, hence we only consider the case $\phi_1*\phi_2\colon G_1*G_2\rightarrow A_1\times A_2$. Let $\tilde{T}$ be the Bass-Serre covering tree of $G_1*G_2$. $\tilde{T}/G_1*G_2$ has two vertices labeled by $G_1$ and $G_2$ and there exists one edge labeled by the trivial group connecting them. We consider the graph $\tilde{T}/N$, where $N$ denotes the kernel of the map $\phi_1*\phi_2$. $\tilde{T}/N$ has $|A_1||A_2|$ edges, where $|A_1||A_2|$ is the index of $N$ in $G_1*G_2$. As for vertices, $\tilde{T}/N$ has $|A_2|$ vertices labeled by $N_1$ and $|A_1|$ vertices labeled by $N_2$. Hence, the fundamental group of $\tilde{T}/N$ as a topological space is the free group of rank $|A_1||A_2|-|A_1|-|A_2|+1=(|A_1|-1)(|A_2|-1)$. The triviality of all the edge stabilizers completes the proof.
\end{proof}
We can characterize the shape of $\Gamma$ such that the underlying graph of the subgroup $K$ appeared in Theorem \ref{Embedding} can be inductively computed by using Lemma \ref{FP}:
\begin{lem}\label{noP3}
Let  $\mathcal{P}$ ($\ni\mathbb{Z}$) be the smallest class of groups closed under taking free and direct products. Then $\mathcal{P}$ coincides with the family of RAAGs whose underlying graphs do not have $P_4$ as their full subgraphs. Here $P_4$ denotes the path on four vertices.
\end{lem}
\begin{proof}
It is clear that $\mathcal{P}$ is a subclass of the class of RAAGs, hence we show that a graph $\Gamma$ having no $P_4$ is made from a single vertex by taking simplicial joins and disjoint unions finite times. (The opposite direction is clear.) The proof proceeds by induction on the number of vertices and we may assume that $\Gamma$ and its complement graph $\bar{\Gamma}$ are connected. This is because $\Gamma=A*B$ if and only if $\bar{\Gamma}=\bar{A}\sqcup\bar{B}$ and $\Gamma$ has no $P_4$ if and only if $\bar{\Gamma}$ has no $P_4$. When $|\text{V}(\Gamma)|=1$, the claim is trivial. Now assume $|\text{V}(\Gamma)|>1$ and let $v$ be a vertex of $\Gamma$, then from the hypothesis of the assumption, $\Gamma-\{v\}$ decomposes as a disjoint union or as a simplicial join. 

\noindent{\bf Case 1 ($\Gamma-\{v\}$ decomposes as a disjoint union: $\Gamma-\{v\}=A_1\sqcup A_2$)} It is enough to show that $\Gamma={v}*(\Gamma-\{v\})$. Assume on the contrary that there exists a vertex $w\in A_i$ not adjacent to $v$. Then since $\Gamma$ is connected, there exists $w'\in A_i$ such that $d(v,w')=2$. On the other hand, there exists a vertex $w'' \in A_j $($j$ is different from $i$) adjacent to $v$. We have $d(w',w'')=3$ hence $\Gamma$ has $P_4$, which is a contradiction.

\noindent{\bf Case 2 ($\Gamma-\{v\}$ decomposes as a simplicial join: $\Gamma-\{v\}=A_1* A_2$)} Consider the complement graph of $\Gamma-\{v\}$, which is a disjoint union of $\bar{A_1}$ and $\bar{A_2}$; this is the case proved just before.
\end{proof}
Graphs not having $P_4$ are known as {\it cographs} or as {\it $P_4$-free graphs} and are classically well known objects in the field of graph theory. The lemma above indicates that outer automorphism groups of RAAGs associated to $P_4$-free graphs can be embedded in outer automorphism groups of RAAGs, whose defining graphs can be computed inductively. 
\begin{cor}\label{FPA}
Let $r_i\in\mathbb{Z}_{>0}$ and $e_i\in\mathbb{Z}_{\geq0}$ be given for each positive integer $i$. Set $R:=\prod_i r_i^{ie_i}$ and $E:=\sum_i e_i$ and assume that $E<\infty$ and $r_1$ is divisible by $r_i$ for every $i$ such that $e_i>0$. Then
\begin{displaymath}
\text{Out}\left(\mathop{*}_{i\geq1}(\mathbb{Z}^i)^{*e_i}\right)\hookrightarrow\text{Out}\left(F_{(E-1)R-\sum_{i}\frac{r_iR}{e_i^i}+1}*\mathop{*}_{i\geq1}(\mathbb{Z}^i)^{*\frac{e_iR}{r_i^i}}\right),
\end{displaymath}
if $(E-1,r)=1$ in the case $e_1>0$ and if $(e_1,e_2,\cdots,e_i-1,\cdots,r_i)=1$ for every $i$ such that $e_i>0$ in the case $e_1=0$.
\end{cor}
\begin{proof}
Set $\mathop{*}_{i}^{*}(\mathbb{Z}^i)^{*e_i}=A_\Gamma$. Here, $\Gamma$ is a disjoint union of $E$ cliques.
It is not hard to see that the kernel $L=\text{Ker}(\mathop{*}_{i}^{*}(\mathbb{Z}^i)^{*e_i}\twoheadrightarrow\prod_{i}(\mathbb{Z}_{r_i}^i)^{e_i})$ is a characteristic subgroup of $A_\Gamma$ if and only if $e_1=0$ or $r_i$ divides $r_1$ for every $i$. (A subgroup $H$ of a group $G$ is characteristic in $G$ if and only if $\phi^\text{ab}(\pi^\text{ab}(H))=\pi^\text{ab}(H)$ for all $\phi\in\text{Aut}(G)$, where $\pi^\text{ab}\colon G\rightarrow G^\text{ab}$ denotes the abelianization map of $G$ and $\phi^\text{ab}$ denotes the automorphism of $G^\text{ab}$ induced by $\phi$. We handle the case where $G$ is a RAAG and, as we have seen in Section 2, we know that the Laurence-Servatius generators generate $\text{Aut}(G)$.) We can compute the form of $L$ using Lemma \ref{FP}:
\[
L\cong F_{(E-1)R-\sum_{i}\frac{r_iR}{e_i^i}+1}*\mathop{*}_{i\geq1}(\mathbb{Z}^i)^{*\frac{e_iR}{r_i^i}}.
\]
Since $K\supset P_{r_1}(A_\Gamma)$, the case $e_1>0$ is just a consequence of Corollary \ref{FF}. As for the case $e_1=0$, we use Theorem \ref{Embedding}. Since there are no non-adjacent domination relations, the first, second and fourth conditions in Theorem \ref{Embedding} are automatically satisfied. From the third condition, it is required that $s_{v,A}$ only depends on the rank $j$ of the maximal free abelian group that $v$ belongs to, and the rank $k$ of $A_A$. In this notation of $j$ and $k$, denote $s_{v,A}=s_{j,k}$. Now, from the fifth condition of Theorem \ref{Embedding}, it is enough to show there exist integers $\{s_{j,k}\}$ satisfying, for each $i$,
$$
\sum_{j}e_js_{j,k}-e_i\equiv-1\pmod{r_i},
$$
which holds if the assumption of Corollary \ref{FPA} is satisfied.
\end{proof}

We can similarly construct embeddings between outer automorphism groups of direct products of (non-abelian) free groups.
\begin{cor}\label{DPF}
Let $r_i\in\mathbb{Z}_{>0}$ and $e_i\in\mathbb{Z}_{\geq0}$ be given for each integer $i\geq2$. Assume that $\sum_i e_i<\infty$. Then
\begin{displaymath}
\text{Out}\left(\prod_{i\geq2}F_i^{e_i}\right)\hookrightarrow\text{Out}\left(\prod_{i\geq2}F_{r_i^i(i-1)+1}^{e_i}\right),
\end{displaymath}
if $(i-1,r_i)=1$ for every $i$ such that $e_i>0$.
\end{cor}

We end this section with several examples.
\begin{exa}
\begin{align*}
\text{Out}(\mathbb{Z}^2*\mathbb{Z})&\hookrightarrow\text{Out}((\mathbb{Z}^2)^{*2}*F_7)\\
\text{Out}(\mathbb{Z}^2*\mathbb{Z})&\hookrightarrow\text{Out}((\mathbb{Z}^2)^{*3}*F_{25})\\
\text{Out}(\mathbb{Z}^2*\mathbb{Z}^2)&\hookrightarrow\text{Out}((\mathbb{Z}^2)^{*8}*F_9)\\
\text{Out}(\mathbb{Z}^2*F_2)&\hookrightarrow\text{Out}((\mathbb{Z}^2)^{*9}*F_{154})\\
\text{Out}(F_2\times F_2)&\hookrightarrow\text{Out}(F_5\times F_5)\\
\text{Out}(F_2\times F_2)&\hookrightarrow\text{Out}(F_{10}\times F_{10})\\
\text{Out}(F_2\times F_3)&\hookrightarrow\text{Out}(F_{10}\times F_{55})
\end{align*}
\end{exa}

\section{Non-embeddability conditions}
Finite subgroups in a group $G$ often become obstructions to embed the group $G$ into another group. Hence in order to establish the non-embeddability between groups, it is important to investigate finite subgroups of those groups. For example, by investigating Boolean subgroups in outer automorphism groups of RAAGs, Kielak \cite{K} lately showed that $\text{Out}(A_\Gamma)$ can not be embedded into $\text{Out}(A_\Lambda)$ if $|\text{V}(\Gamma)|>|\text{V}(\Lambda)|$:
\begin{theo}\cite[Theorem 4.1]{K}\label{Kielak}
There are no injective homomorphisms $\text{Out}(A_\Gamma)\rightarrow\text{Out}(A_{\Gamma'})$, where $\Gamma'$ has fewer vertices than $\Gamma$.
\end{theo}

In this section, we investigate finite subgroups in pure (outer) automorphism groups of RAAGs and give necessary conditions of embeddings.
\subsection{Pure automorphism groups}
We begin from the definition of the pure (outer) automorphism groups of RAAGs. Recall that the automorphism group $\text{Aut}(A_\Gamma)$ is generated by graph symmetries, inversions, transvections and partial conjugations. The {\it pure automorphism group} $\text{Aut}^0(A_\Gamma)$ is the subgroup of $\text{Aut}(A_\Gamma)$ generated by inversions, transvections and partial conjugations. Notice that the inner automorphism group $\text{Inn}(A_\Gamma)$ is contained in $\text{Aut}^0(A_\Gamma)$ since every inner automorphism is a product of several number of partial conjugations. We define the {\it pure outer automorphism group} $\text{Out}^0(A_\Gamma)$ to be the quotient group $\text{Aut}^0(A_\Gamma)/\text{Inn}(A_\Gamma)$. Note that the pure (outer) automorphism group of $A_\Gamma$ coincides with the (outer) automorphism group of $A_\Gamma$ when $\Gamma$ is complete or null.

\subsection{Structure of automorphism groups of simple graphs}
Recall that $\text{V}(\Gamma)$ is equipped with an equivalence relation $\sim$: $v\sim w$ if and only if $\text{lk}(v)\subset\text{st}(w)$ and $\text{st}(v)\supset\text{lk}(w)$.
\begin{lem}\label{wd}
Every $[v]\in\text{V}(\Gamma)/\sim$ spans either a complete subgraph or a null subgraph.
\end{lem}
\begin{proof}
Assume on the contrary that there exist vertices $v_1$, $v_2$ and $v_3$ in $[v]$ such that $\{v_1,v_2\}\in\text{E}(\Gamma)$ and $\{v_1,v_3\}\notin\text{E}(\Gamma)$. Then $v_1\in\text{lk}(v_2)$ but $v_1\notin\text{st}(v_3)$. This contradicts the assumption that $v_2\sim v_3$.
\end{proof}
\begin{lem}\label{wd2}
Let $[v]$ and $[w]$ be elements in $\text{V}(\Gamma)$. Then we have,
$$
\{v,w\}\in\text{E}(\Gamma)\Leftrightarrow\{v',w'\}\in\text{E}(\Gamma),\ \text{for all } v,v'\in[v]\text{ and for all } w,w'\in[w].
$$
\end{lem}
\begin{proof}
The case $[v]=[w]$ is Lemma \ref{wd}, hence we assume $[v]\neq[w]$. 
Since $w\in\text{lk}(v)\subset\text{st}(v')$, we have $\{v',w\}\in\text{E}(\Gamma)$. Hence $v'\in\text{lk}(w)\subset\text{st}(w')$.
\end{proof}
From Lemma \ref{wd2}, the quotient $\bar{\Gamma}$ of the original graph $\Gamma$ by the equivalence relation $\sim$ on $\text{V}(\Gamma)$ is well-defined as a simple graph: $\text{V}(\bar{\Gamma})=\text{V}(\Gamma)/\sim$ and $\{[v],[w]\}\in\text{E}(\bar{\Gamma})$ if and only if $\{v,w\}\in\text{E}(\Gamma)$ and $[v]\neq[w]$. We can view $\bar{\Gamma}$ as a vertex-labeled graph: $[v]$ is labeled by the RAAG whose underlying graph is the subgraph of $\Gamma$ spanned by $[v]$. From Lemma \ref{wd}, labels are free groups or free abelian groups.
\begin{lem}
$$
\text{Aut}(\Gamma)=\left(\prod_{[v]\in\text{V}(\Gamma)/\sim}S_{|[v]|}\right)\rtimes\text{Aut}^{\text{lab}}(\bar{\Gamma}),
$$
where $\text{Aut}^{\text{lab}}(\bar{\Gamma})$ denotes the group of automorphisms of $\bar{\Gamma}$ which preserve the labels of vertices.
\end{lem}
\begin{proof}
Note that every graph automorphism sends each equivalence class to one equivalence class, that is, $\sigma([v])=[\sigma(v)]$. The set of automorphisms which fix every equivalence class forms a group $S$. $S$ is isomorphic to $\prod_{[v]\in\text{V}(\Gamma)/\sim}S_{|[v]|}$ since, from Lemma \ref{wd2}, vertices may be permuted in arbitrary way in each class. Then define $Q$ to be the subgroup of $\text{Aut}(\Gamma)$ which consists of elements permuting equivalence classes but fixing orders in all classes. $Q$ is clearly isomorphic to $\text{Aut}^{\text{lab}}(\bar{\Gamma})$. Now, since $S\triangleleft\text{Aut}(\Gamma)$, $S\cap Q$ is trivial and $SQ=\text{Aut}(\Gamma)$, we have $\text{Aut}(\Gamma)=S\rtimes Q$.
\end{proof}
This lemma provides $\text{Aut}(A_\Gamma)$ with the structure of a semi-direct product of $\text{Aut}^0(A_\Gamma)$ and $\text{Aut}^{\text{lab}}(\bar{\Gamma})$ since the intersection of $\text{Aut}^0(A_\Gamma)$ and $\text{Aut}(\Gamma)$ is $\prod_{[v]\in\text{V}(\Gamma)/\sim}S_{|[v]|}$:
\begin{cor}\label{straut}
\begin{align*}
\text{Aut}(A_\Gamma)&=\text{Aut}^0(A_\Gamma)\rtimes\text{Aut}^{\text{lab}}(\bar{\Gamma}), \\
\text{Out}(A_\Gamma)&=\text{Out}^0(A_\Gamma)\rtimes\text{Aut}^{\text{lab}}(\bar{\Gamma}).
\end{align*}
\end{cor}

\subsection{Finite subgroups in pure (outer) automorphism groups of right-angled Artin groups}
We introduce some notations to evaluate torsion subgroups in groups:
\begin{itemize}
\item $\nu_p(G)$ denotes the {\it $p$-adic valuation} of $G$, that is, $\max\{t\mid \exists H<G,\text{ s.t. } |H|=p^t\}$,
\item $\text{rank}_{\mathbb{Z}_p}(G)$ denotes the {\it $\mathbb{Z}_p$-rank} of $G$, that is, $\max\{t\mid \mathbb{Z}_p^t\hookrightarrow G\}$.
\end{itemize}
Note that, for a finite group $G$, $\nu_p(G)=\nu_p(|G|)$ due to Sylow's theorem, where $\nu_p$ in the right hand side of the equality denotes the commonly-used $p$-adic valuation of integers.
We offer some properties of these values:
\begin{lem}\label{basictools}
Let $G$ and $H$ be groups. Then,
\begin{itemize}
\item $0\leq\text{rank}_{\mathbb{Z}_p}(G)\leq\nu_p(G)\leq\infty$
\item $\nu_p(G\times H)=\nu_p(G)+\nu_p(H)$ and $\text{rank}_{\mathbb{Z}_p}(G\times H)=\text{rank}_{\mathbb{Z}_p}(G)+\text{rank}_{\mathbb{Z}_p}(H)$, 
\item $\nu_p(G*H)=\max\{\nu_p(G),\nu_p(H)\}$ and $\text{rank}_{\mathbb{Z}_p}(G*H)=\max\{\text{rank}_{\mathbb{Z}_p}(G),\text{rank}_{\mathbb{Z}_p}(H)\}$,
\item $\max\{\nu_p(G),\nu_p(H)\}\leq\nu_p(G\rtimes H)\leq\nu_p(G)+\nu_p(H)$ and $\max\{\text{rank}_{\mathbb{Z}_p}(G),\text{rank}_{\mathbb{Z}_p}(H)\}\leq\text{rank}_{\mathbb{Z}_p}(G\rtimes H)\leq\text{rank}_{\mathbb{Z}_p}(G)+\text{rank}_{\mathbb{Z}_p}(H)$,
\item if there exists a homomorphism $\phi\colon G\rightarrow H$ whose kernel is torsion-free, then $\nu_p(G)\leq\nu_p(H)$ and $\text{rank}_{\mathbb{Z}_p}(G)\leq\text{rank}_{\mathbb{Z}_p}(H)$.
\end{itemize}
\end{lem}
Notice that there is a monotonicity of $\nu_p$ and $\text{rank}_{\mathbb{Z}_p}$ with respect to inclusions of groups, as the final statement of Lemma \ref{basictools} indicates.
$p$-adic valuations and $\mathbb{Z}_p$-ranks of (outer) automorphism groups of free groups and general linear groups over integers are computed in the next lemma.
\begin{lem}\label{autoutp}
Let $p$ be a prime and let $n$ be a positive integer. Then,
\begin{align*}
\nu_p(\text{Aut}(F_n))&=\nu_p(\text{Out}(F_n))=\nu_p(\text{GL}(n,\mathbb{Z}))=\sum_{k\geq0}\left\lfloor\frac{n}{p^k(p-1)}\right\rfloor,\\
\text{rank}_{\mathbb{Z}_p}(\text{Aut}(F_n))&=\text{rank}_{\mathbb{Z}_p}(\text{Aut}(F_n))=\text{rank}_{\mathbb{Z}_p}(\text{GL}(n,\mathbb{Z}))=\left\lfloor\frac{n}{p-1}\right\rfloor.
\end{align*}
\end{lem}
\begin{proof}
Values other than $\nu_p(\text{GL}(n,\mathbb{Z}))$ and $\text{rank}_{\mathbb{Z}_p}(\text{GL}(n,\mathbb{Z}))$ are almost obvious from Theorem \ref{Kh} and Theorem \ref{Ma}.

\noindent{\bf $p$-adic valuation} From Theorem \ref{Minkowski}, since $\text{GL}(n,\mathbb{Z})<\text{GL}(n,\mathbb{Q})$,
\[
\nu_p(\text{GL}(n,\mathbb{Z}))\leq\sum_{k\geq0}\left\lfloor\frac{n}{p^k(p-1)}\right\rfloor.
\] 
On the other hand, Baumslag and Taylor \cite{B&T} showed that the natural maps below have torsion-free kernels.
\[
\text{Aut}(F_n)\rightarrow\text{Out}(F_n)\rightarrow\text{GL}(n,\mathbb{Z}).
\]
Thus,
\[
\nu_p(\text{GL}(n,\mathbb{Z}))\geq\sum_{k\geq0}\left\lfloor\frac{n}{p^k(p-1)}\right\rfloor,
\] 
giving the exact value of $\nu_p(\text{GL}(n,\mathbb{Z}))$.

\noindent{\bf $\mathbb{Z}_p$-rank} Let $\mathbb{Z}_p^r$ be embedded in $\text{GL}(n,\mathbb{Z})$. The character $\rho$ is an $\mathbb{Z}_{\geq0}$-combination of irreducible characters, whose set is denoted by $\hat{\mathbb{Z}_p^r}$. Since $\mathbb{Z}_p^r$ is abelian, $\hat{\mathbb{Z}_p^r}$ coincides with the set of group homomorphisms from $\mathbb{Z}_p^r\text{ to }\mathbb{C}^*$. This set is divided into subsets corresponding to subgroups: $\hat{\mathbb{Z}_p^r}=\bigsqcup_{H<\mathbb{Z}_p^r}\hat{\mathbb{Z}_p^r}_H$, where $\hat{\mathbb{Z}_p^r}_H$ denotes the set of irreducible characters whose kernels are $H$. Denote $\sum_{\chi\in\hat{\mathbb{Z}_p^r}_H}\chi$ by $\chi_H$. As $\rho$ is $\mathbb{Z}$-valued, 
\[
\rho=\sum_{H<\mathbb{Z}_p^r}\alpha_H\chi_H\text{ ($\alpha_H\in\mathbb{Z}_{\geq0}$)}.
\]
Moreover, since $\mathbb{Z}_p^r\hookrightarrow\text{GL}(n,\mathbb{Z})$, $\bigcap_{\alpha_H>0}H$ is trivial. Recall that irreducible characters are group homomorphisms into complex numbers, and hence kernels are cyclic. Thus $[\mathbb{Z}_p^r:H]$ is $p$ except for the case $H=\mathbb{Z}_p^r$. Denote $|\{H\neq\mathbb{Z}_p^r|\alpha_H>0\}|$ by $N$. Then, 
\[
\frac{1}{p^N}|\mathbb{Z}_p^r|\leq 1.
\]
Hence, $N\geq r$. Now we have:
\[
n=\dim\rho=\sum_H\alpha_H|\hat{\mathbb{Z}_p^r}_H|\geq N(p-1)\geq r(p-1),
\]
as desired.
\end{proof}
The next lemma is a key to evaluate finite subgroups in pure (outer) automorphism groups of RAAGs from above.
\begin{lem}\label{above}
All the homomorphisms appearing in the diagram below have torsion-free kernels:
$$
\text{Aut}^0(A_\Gamma)\rightarrow\text{Out}^0(A_\Gamma)\rightarrow\text{TB}(\{|[v]|\},\mathbb{Z})\rightarrow\text{DB}(\{|[v]|\},\mathbb{Z})\rightarrow\prod_{[v]\in\text{V}(\Gamma)/\sim}\text{GL}(|[v]|,\mathbb{Z}),
$$
where $\text{TB}(\{|[v]|\},\mathbb{Z})$ denotes the group of (upper) triangular block invertible matrices, where the basis of vector space are divided into subsets which have sizes $\{|[v]|\}$, and $\text{DB}(\{|[v]|\},\mathbb{Z})<\text{TB}(\{|[v]|\},\mathbb{Z})$ denotes the group of diagonal block invertible matrices.
\end{lem}
\begin{proof}

\noindent{\bf The first map} The kernel of this homomorphism is $\text{Inn}(A_\Gamma)$, which is torsion-free since it is also a RAAG $A_\Lambda$. This is an easy consequence of Servatius' Centralizer Theorem \cite{S}.

\noindent{\bf The second map} The fact that the kernel of the map $\text{Out}(A_\Gamma)\rightarrow\text{GL}(|\text{V}(\Gamma)|,\mathbb{Z}); \phi\mapsto\phi^{\text{ab}}$ is torsion-free is proved by Toinet \cite{T}. By ordering vertices appropriately, we can make the image of the pure outer automorphism group $\text{Out}^0(A_\Gamma)$ contained in $\text{TB}(\{|[v]|\},\mathbb{Z})$: vertex domination induces a partial order on $\text{V}(\Gamma)/\sim$ and hence one can assign integers $1,\cdots,|\text{V}(\Gamma)/\sim|$ so that the number of $[w]$ is bigger than the number of $[v]$ if $[v]<[w]$.

\noindent{\bf The third map} The third map is the projection onto diagonal blocks. The kernel of this map consists only of unitriangular matrices, that is, triangular matrices whose diagonal entries are all $1$. Thus the kernel of the third map is torsion-free since the group of unitriangular matrices is torsion-free.

\noindent{\bf The fourth map} This map is an isomorphism
.\end{proof}

\begin{cor}
Let $\Gamma$ be an asymmetric graph, that is, a graph whose automorphism group is trivial and let the (outer) automorphism group of $A_\Lambda$ be embedded in the (outer) automorphism group of $A_\Gamma$. Then $\Lambda$ is also asymmetric.
\end{cor}
\begin{proof}
From Lemma \ref{above}, every finite subgroup $G$ of the (outer) automorphism group of $A_\Gamma$ are embedded into the group of upper triangular matrices. Hence each element in $G$ has order two and thus $G$ is abelian.
Now assume on the contrary that $\text{Aut}(\Lambda)$ is not trivial. Then the group of type (1) Whitehead automorphisms, which is isomorphic to $\mathbb{Z}_2^{|\text{V}(\Lambda)|}\rtimes\text{Aut}(\Lambda)$, is not abelian since $\text{Aut}(\Lambda)$ acts on $\mathbb{Z}_2^{|\text{V}(\Lambda)|}$ by basis permutations. This group is an obstruction to embed the (outer) automorphism group of $A_\Lambda$ into the (outer) automorphism group of $A_\Gamma$.
\end{proof}
In order to evaluate from below, the next lemma is useful.
\begin{lem}\label{below}
Let $F\subset\text{V}(\Gamma)/\sim$ (respectively, $A\subset\text{V}(\Gamma)/\sim$) be the set of elements which span null graphs with more than one vertex (respectively, complete graphs). Then there is an injection
$$
\prod_{[v]\in F}\text{Aut}(F_{|[v]|})\times\prod_{[v]\in A}\text{GL}(|[v]|,\mathbb{Z})\hookrightarrow\text{Aut}^0(A_\Gamma).
$$
\end{lem}
\begin{proof}
Given automorphisms of subgroups whose underlying subgraphs are spanned by equivalence classes, they can be extended to a group homomorphism defined on the the whole group $A_\Gamma$.
\end{proof}
\begin{theo}\label{finsub}
\begin{align*}
\nu_p(\text{Aut}^0(A_\Gamma))&=\nu_p(\text{Out}^0(A_\Gamma))=\sum_{[v]\in\text{V}(\Gamma)/\sim,k\geq0}\left\lfloor{\frac{|[v]|}{p^k(p-1)}}\right\rfloor\\
\text{rank}_{\mathbb{Z}_p}(\text{Aut}^0(A_\Gamma))&=\text{rank}_{\mathbb{Z}_p}(\text{Out}^0(A_\Gamma))=\sum_{[v]\in\text{V}(\Gamma)/\sim}\left\lfloor{\frac{|[v]|}{p-1}}\right\rfloor
\end{align*}
\end{theo}
\begin{proof}
This is immediate from Lemmas \ref{basictools}, \ref{above} and \ref{below}.
\end{proof}
\begin{cor}\label{finsubcor}
\begin{align*}
\sum_{[v],k}\left\lfloor{\frac{|[v]|}{p^k(p-1)}}\right\rfloor&\leq\nu_p(\text{Aut}(A_\Gamma))\leq\min\left\{\sum_{[v],k}\left\lfloor{\frac{|[v]|}{p^k(p-1)}}\right\rfloor+\nu_p\left(\text{Aut}^{\text{lab}}(\bar{\Gamma})\right),\sum_{k\geq0}\left\lfloor{\frac{|\text{V}(\Gamma)|}{p^k(p-1)}}\right\rfloor\right\}\\
\sum_{[v],k}\left\lfloor{\frac{|[v]|}{p^k(p-1)}}\right\rfloor&\leq\nu_p(\text{Out}(A_\Gamma))\leq\min\left\{\sum_{[v],k}\left\lfloor{\frac{|[v]|}{p^k(p-1)}}\right\rfloor+\nu_p\left(\text{Aut}^{\text{lab}}(\bar{\Gamma})\right),\sum_{k\geq0}\left\lfloor{\frac{|\text{V}(\Gamma)|}{p^k(p-1)}}\right\rfloor\right\}\\
\sum_{[v]}\left\lfloor{\frac{|[v]|}{p-1}}\right\rfloor&\leq\text{rank}_{\mathbb{Z}_p}(\text{Aut}(A_\Gamma))\leq\min\left\{\sum_{[v]}\left\lfloor{\frac{|[v]|}{p-1}}\right\rfloor+\text{rank}_{\mathbb{Z}_p}\left(\text{Aut}^{\text{lab}}(\bar{\Gamma})\right),\left\lfloor{\frac{|\text{V}(\Gamma)|}{p-1}}\right\rfloor\right\}\\
\sum_{[v]}\left\lfloor{\frac{|[v]|}{p-1}}\right\rfloor&\leq\text{rank}_{\mathbb{Z}_p}(\text{Out}(A_\Gamma))\leq\min\left\{\sum_{[v]}\left\lfloor{\frac{|[v]|}{p-1}}\right\rfloor+\text{rank}_{\mathbb{Z}_p}\left(\text{Aut}^{\text{lab}}(\bar{\Gamma})\right),\left\lfloor{\frac{|\text{V}(\Gamma)|}{p-1}}\right\rfloor\right\}
\end{align*}
\end{cor}
\begin{proof}
The inequalities of left hand sides and the first terms of $\min$ in the right hand sides are clear from Theorem \ref{finsub} and Corollary \ref{straut} and Lemma \ref{basictools}. As for the second term, recall that kernels of maps below are torsion-free:
\[
\text{Aut}(A_\Gamma)\rightarrow\text{Out}(A_\Gamma)\rightarrow\text{GL}(|\text{V}(\Gamma)|,\mathbb{Z}),
\]
as also stated in the proof of Theorem \ref{above}. Combining this with Lemma \ref{basictools} and Lemma \ref{autoutp} completes the proof. 
\end{proof}
Note that the result of Kielak (Theorem \ref{Kielak}) is immediate if we apply the fourth formula of Corollary \ref{finsubcor} to the case $p=2$.
\begin{cor}
Let $\nu_p(\text{Aut}(\Gamma))<\nu_p(\text{Aut}(\Lambda))$ and $|[v]|<p-1$ for all $[v]\in\text{V}(\Gamma)/\sim$. Then there are no embeddings from the (outer) automorphism group of $A_\Lambda$ to the (outer) automorphism group of $A_\Gamma$.
\end{cor}
\begin{proof}
Assume on the contrary that a $p$-group of order $p^{\nu_p(\text{Aut}(\Lambda))}$ are embedded into the (outer) automorphism group of $A_\Gamma$. Then the intersection of its image and the pure (outer) automorphism group of $A_\Gamma$ has order $p^t$, where $t\geq1$. Now applying Corollary \ref{finsub} completes the proof.
\end{proof}

\end{document}